\documentclass{amsart}
\usepackage{amsthm,amsmath,amssymb,comment,enumerate,amsfonts,mathtools,afterpage,hyperref,todonotes,algorithm,algorithmic,multicol,tikz-cd,subcaption,float,bbm,epic}
\usepackage[nameinlink]{cleveref}
\usepackage[utf8]{inputenc}

\newtheorem{theorem}{Theorem}[section]
\newtheorem{def-prop}[theorem]{Definition-Proposition}
\newtheorem{prop}[theorem]{Proposition}
\newtheorem{conj}[theorem]{Conjecture}
\newtheorem{lemma}[theorem]{Lemma}
\newtheorem{cor}[theorem]{Corollary}

\theoremstyle{definition}

\newtheorem{defin}[theorem]{Definition}

\theoremstyle{remark}

\DeclareMathOperator{\BZ}{BZ}
\DeclareMathOperator{\wt}{wt}
\DeclareMathOperator{\Supp}{Supp}

\newcommand{\R}{\mathbb{R}}
\newcommand{\Yibo}[1]{\todo[size=\tiny,inline,color=yellow!30]{#1 \\ \hfill --- Yibo}}
\title{Degrees of the stretched Kostka quasi-polynomials}
\author{Shiliang Gao}
\address{Department of Mathematics, University of Illinois at Urbana-Champaign, Urbana, IL 61801}
\email{\href{mailto:sgao23@illinois.edu}{{\tt sgao23@illinois.edu}}}

\author{Yibo Gao}
\address{Department of Mathematics, University of Michigan, \mbox{Ann Arbor, MI 48109}}
\email{\href{mailto:gaoyibo@umich.edu}{{\tt gaoyibo@umich.edu}}}

\date{\today}

\begin{document}

\begin{abstract}
We provide a type-uniform formula for the degree of the stretched Kostka quasi-polynomial $K_{\lambda,\mu}(N)$ in all classical types, improving a previous result by McAllister in $\mathfrak{sl}_r(\mathbb{C})$. Our proof relies on a combinatorial model for the weight multiplicity by Berenstein and Zelevinsky.
\end{abstract}
\maketitle

\section{Introduction} \label{sec:intro}
The Kostka numbers $K_{\lambda,\mu}$ have great importance in many branches of mathematics, including representation theory, algebraic geometry, symmetric function theory, polytope theory and so on. They can be defined as the dimension of the weight subspace of weight $\mu$ in the irreducible representation $V_{\lambda}$ with highest weight $\lambda$. In type $A$ (i.e. $\mathfrak{sl}_r(\mathbb{C})$), they are also well-known as the coefficient of the expansion of Schur functions $s_{\lambda}$ into monomial symmetric functions $m_{\mu}$, the number of semistandard Young tableaux of shape $\lambda$ and weight $\mu$, the number of lattice points in the Gelfand-Tsetlin polytope $\mathrm{GT}_{\lambda}$ of weight $\mu$, and many others \cite{YoungTableaux,EC2}.

One method to study a family of numbers is to organize them together. Define the \emph{stretched Kostka quasi-polynomial} via $K_{\lambda,\mu}(N)=K_{N\lambda,N\mu}$ (Definition~\ref{def:kostka-quasi}), which is known to be a quasi-polynomial. In type $A$, it is a polynomial and can be investigated via Ehrhart theory on the Gelfand-Tsetlin polytope, which enjoys rich combinatorial and polyhedral geometric properties including its face structures \cite{an2018fvectors,deloera2004vertices,mcallister2006thesis}, its connection to order polytopes \cite{ardila2011Gelfand,stanley1986two} , its combinatorial diameter and automorphism group \cite{gao2019diameter}. In particular, McAllister utilized the combinatorics of the Gelfand-Tsetlin polytope to compute the degree of the stretched Kostka polynomial in type $A$ \cite{M08}, resolving a conjecture by King, Tollu, and Toumazet \cite{KTT04}. However, although an analogue of Gelfand-Tsetlin polytopes is defined in other classical types \cite{BZ88}, which we call \emph{Berenstein-Zelevinsky} polytope, its various properties are not so well-understood.

In this paper, we provide a type-uniform formula for the degree of the stretched Kostka quasi-polynomial $K_{\lambda,\mu}(N)$ in all classical types (Theorem~\ref{thm:main}), improving a result by McAllister \cite{M08}. Along the way, we study relevant properties of the corresponding Berenstein-Zelevinsky polytope, especially its interior points. In the remainder of the section, we state our main result (Theorem~\ref{thm:main}). We provide further background in Section~\ref{sec:background} and prove the main result in Section~\ref{sec:proof}.

We adopt the following notations throughout the paper. More background will be provided in Section~\ref{sec:background}. Let $\mathfrak{g}$ be a complex semisimple Lie algebra with Cartan subalgebra $\mathfrak{h}$ and root system $\Phi\subset\mathfrak{h}^{*}$. Let $r=\dim\mathfrak{g}$, which is also the rank of $\Phi$, denoted $\mathrm{rk}(\Phi)$. Let $\Phi_+$ be the positive roots,  $\Pi=\{\alpha_1,\ldots,\alpha_r\}$ be the simple roots and $\{\omega_1,\ldots,\omega_r\}$ be the corresponding fundamental weights. For dominant weights $\lambda$ and $\mu$, let $V_{\lambda}$ be the irreducible finite-dimensional representation of $\mathfrak{g}$ with highest weight $\lambda$, and let $V_{\lambda}(\mu)\subset V_{\lambda}$ be its subspace of weight $\mu$. 
\begin{defin}\label{def:kostka-quasi}
The \emph{stretched Kostka quasi-polynomial} (in $N$) is defined as \[K_{\lambda,\mu}(N):=\dim V_{N\lambda}(N\mu).\] 
\end{defin}
It is known that $K_{\lambda,\mu}(N)$ is a quasi-polynomial in general, which is a polynomial in type $A$. It is also a classical result that $\dim V_{\lambda}(\mu)\neq0$ if and only if $\lambda-\mu$ is a nonnegative integral linear combination of the simple roots \cite{BZ88}.

\begin{theorem}\label{thm:main}
Let $\mathfrak{g}$ be a complex semisimple Lie algebra of classical type and dimension $r$. For dominant weights $\lambda,\mu\in\mathfrak{h}^{*}$ such that $\lambda-\mu=\sum_{i=1}^r c_i\alpha_i$ with $c_i\in\mathbb{Z}_{\geq0}$ for all $i$, write $\lambda=\sum_{i=1}^r d_i\omega_i$ where $d_i\in\mathbb{Z}_{\geq0}$ for all $i$. Let
\[\begin{cases}
\Phi^{(1)}\subset\Phi\text{ be the root subsystem spanned by }\{\alpha_i\:|\: c_i\neq0\},\\
\Phi^{(2)}\subset\Phi^{(1)}\text{ be the root subsystem spanned by }\{\alpha_i\:|\: c_i\neq0,d_i=0.\}
\end{cases}\]
Then the degree of the stretched Kostka quasi-polynomial equals
\[\deg K_{\lambda,\mu}(N)=|\Phi^{(1)}_+|-\mathrm{rk}(\Phi^{(1)})-|\Phi^{(2)}_+|.\]
\end{theorem}

\begin{conj}\label{conj:main}
Theorem~\ref{thm:main} holds for all complex semisimple Lie algebras (of arbitrary types). 
\end{conj}

\section{Background}\label{sec:background}
We discuss more background on root systems, and then talk about a model in \cite{BZ88} that allows us to compute $K_{\lambda,\mu}(N)$ in classical types.
\subsection{Root system and Weyl groups}
Following Section~\ref{sec:intro}, we continue with more background on root systems. Recall that the positive roots $\Pi=\{\alpha_1,\ldots,\alpha_r\}$ form a basis of our ambient vector space $\mathfrak{h}^{*}$, which is also called the \emph{weight space}. It is sometimes more convenient to identify this space with its dual, by equipping it with an inner product $\langle-,-\rangle$. The \emph{simple coroots} are defined as \[\alpha_i^{\vee}=\frac{2\alpha_i}{\langle\alpha_i,\alpha_i\rangle}\]
for $i=1,\ldots,r$. And the fundamental weights $\{\omega_1,\ldots,\omega_r\}$ are defined as the dual basis to the coroots, i.e. $\langle\alpha_i^{\vee},\omega_j\rangle=\delta_{i,j}$ where $\delta$ is the Kronecker delta. A weight $\lambda\in\mathfrak{h}^*$ is an \emph{integral weight} if $\lambda$ is written as $\sum_{i=1}^rd_i\omega_i$ where $d_i\in\mathbb{Z}$ for all $i$. All integral weights form the \emph{weight lattice} $Q$. We also say that a weight $\lambda=\sum_{i=1}d_i\omega_i$ is an \emph{dominant integral weight} if $d_i\in\mathbb{Z}_{\geq0}$ for all $i$. Let $\Lambda^+$ be the set of all dominant integral weights.  

We define a partial order, called the \emph{root poset}, on the positive roots $\Phi_+$, where $\alpha\leq\beta$ if $\beta-\alpha$ is a nonnegative (integral) linear combination of the simple roots $\Pi$. Thus, the minimal elements of the root poset $\Phi_+$ are precisely the simple roots. Define the \emph{support} of a positive root $\alpha$ to be \[\Supp(\alpha):=\{\alpha_i\in\Pi\:|\: \alpha_i\leq \alpha\}\subset\Pi.\]
We also say that a positive root $\alpha$ is \emph{supported on} $\alpha_i\in\Pi$ if $\alpha_i\in\Supp(\alpha)$. 

Similarly, for $\lambda,\mu\in \Lambda^+$, we say that $\lambda$ \emph{dominates} $\mu$, if
\begin{equation}\label{eqn:lambda-mu}
    \lambda-\mu = \sum_{i=1}^r c_i\alpha_i,
\end{equation}
for some $c_i\in\mathbb{Z}_{\geq0}$. 
We say that a pair $(\lambda,\mu)$ is \emph{primitive} if $c_i\in\mathbb{Z}_{>0}$ for all $i$ in \eqref{eqn:lambda-mu}. We call a primitive pair \emph{simple} if the underlying Lie algebra is simple. 

We adopt the following conventions on the root systems for classical types.
\begin{itemize}
\item Type $A_{r-1}$ ($\mathfrak{g} = \mathfrak{sl}_r$): $\Phi=\{e_i-e_j\:|\: 1\leq i<j\leq r\}$, $\Phi_+=\{e_i-e_j\:|\: 1\leq i<j\leq r\}$, $\alpha_i=e_i-e_{i+1}$ and $\omega_i=e_1+\cdots+e_i$ for $i=1,\ldots,r-1$. 
\item Type $B_r$ ($\mathfrak{g} = \mathfrak{so}_{2r+1}$): $\Phi=\{\pm e_i\pm e_j\:|\:1\leq i<j\leq r\}\cup\{\pm e_i\:|\: 1\leq i\leq r\}$, $\Phi_+=\{e_i\pm e_j\:|\: 1\leq i<j\leq r\}\cup\{e_i\:|\: 1\leq i\leq r\}$, $\alpha_i=e_i-e_{i+1}$ and $\omega_i = e_1+\cdots+e_i$ for $i=1,\ldots,r-1$, and $\alpha_r=e_r$ and $\omega_r=\frac{1}{2}(e_1+\cdots+e_r)$.
\item Type $C_r$ ($\mathfrak{g} = \mathfrak{sp}_{2r}$): $\Phi=\{\pm e_i\pm e_j\:|\:1\leq i<j\leq r\}\cup\{\pm 2e_i\:|\: 1\leq i\leq r\}$, $\Phi_+=\{e_i\pm e_j\:|\: 1\leq i<j\leq r\}\cup\{2e_i\:|\: 1\leq i\leq r\}$, $\alpha_i=e_i-e_{i+1}$ for $i = 1,\ldots,r-1$ and $\alpha_r = 2e_r$, and $\omega_i=e_1+\cdots+e_i$ for $i=1,\ldots,r$.
\item Type $D_r$ ($\mathfrak{g} =\mathfrak{so}_{2r}$): $\Phi=\{\pm e_i\pm e_j\:|\:1\leq i<j\leq r\}$, $\Phi_+=\{e_i\pm e_j\:|\: 1\leq i<j\leq r\}$, $\alpha_i=e_i-e_{i+1}$ for $i=1,\ldots,r-1$, $\alpha_r=e_{r-1}+e_r$, $\omega_i=e_1+\cdots+e_{i}$ for $i=1,\ldots,r-2$, $\omega_{r-1}=\frac{1}{2}(e_1+\cdots+e_{r-1}-e_r)$ and $\omega_{r}=\frac{1}{2}(e_1+\cdots+e_r)$. 
\end{itemize}

\subsection{Kostka coefficients}

Here we review some basic notions and results related to Kostka coefficients. The Kostka coefficients can be computed via the \emph{Kostant multiplicity formula}, and can be expressed as tensor product multiplicities.
\begin{defin}
The \emph{Kostant partition function} is defined by
\begin{equation}
    P(\lambda-\mu) := \#\{(m_{\alpha})_{\alpha\in \Phi^+}:m_{\alpha}\in \mathbb{Z}_{\geq 0},\sum_{\alpha\in \Phi^+}m_{\alpha}\alpha = \lambda-\mu\}.
\end{equation}
\end{defin}

For a pair $\lambda,\mu\in \Lambda^+$ such that $\lambda$ dominates $\mu$,
the \emph{Kostant multiplicity formula} express $K_{\lambda,\mu}$ as alternate sum of Kostant partition functions:
\begin{theorem}\label{thm:Kostantmult}
    $K_{\lambda,\mu} = \sum_{w\in W}\det(w)P(w(\lambda+\rho)-\mu-\rho)$,
where $W$ is the Weyl group and $\rho = \frac{1}{2}\sum_{\alpha\in \Phi^+}\alpha$. 
\end{theorem}


Consider the root space decomposition of $\mathfrak{g}$:
\begin{equation}
    \mathfrak{g} = \mathfrak{h}\oplus \bigoplus_{\alpha\in \Phi}\mathfrak{g}_{\alpha}.
\end{equation}
For $\alpha\in \Pi$, let $e_{\pm\alpha}\in \mathfrak{g}_{\pm\alpha}$ and $h_\alpha = [e_\alpha,e_{-\alpha}]\in \mathfrak{h}$ be the standard generator. 
For each weight space $V_{\lambda}(\beta)\subset V_{\lambda}$ and highest weight $\nu = \sum_{\alpha \in \Pi}n_{\alpha}\alpha$, let $V_{\lambda}(\beta,\nu)\subseteq V_{\lambda}(\beta)$ be the subspace defined by
\begin{equation}\label{eqn:defmult}
    V_{\lambda}(\beta,\nu):= \{v\in V_{\lambda}(\beta):e_{\alpha}^{n_\alpha+1}v = 0\text{ for }\alpha\in \Pi\}.
\end{equation}
Let $c_{\lambda,\nu}^{\mu}$ be the multiplicity of $V_{\mu}$ in the tensor product $V_{\lambda}\otimes V_{\nu}$. It is known (see e.g.\cite{Z73}) that $c^{\mu}_{\lambda,\nu}$ can be realized as the dimension of certain subspace of $V_{\lambda}$:
\begin{equation}\label{eqn:Oct13}
    c^{\mu}_{\lambda,\nu} = \dim(V_\lambda(\mu-\nu,\nu)).
\end{equation}

\begin{lemma}\label{lemma:Kostkatomult}
For $\lambda,\mu\in \Lambda^+$ such that $\lambda$ dominates $\mu$, let $c_i$'s be as in \eqref{eqn:lambda-mu} and set
\[\nu = \sum_{i = 1}^r c_i\omega_i \text{ and }\mu' = \nu+\mu.\]
Then $K_{\lambda,\mu} = c^{\mu'}_{\lambda,\nu}$.
\end{lemma}
\begin{proof}
For any $v\in V_{\lambda}(\mu)$ and any $i\in [r]$,
\[e_{\alpha_i}^{c_{i}+1}v\in V_{\lambda}(\mu+(c_i+1)\alpha_i).\]
Since $\lambda-\mu-(c_i+1)\alpha_i = (\sum_{j\neq i}c_j\alpha_j)-\alpha_i$,
$\lambda$ does not dominate $\mu+(c_i+1)\alpha_i$ for any $i\in [r]$. Therefore 
$$\dim(V_{\lambda}(\mu+(c_i+1)\alpha_i)) = 0 \text{ and }e_{\alpha_i}^{c_{i}+1}v = 0 \text{ for all } i\in [r].$$ 
Now by \eqref{eqn:defmult},
\[V_{\lambda}(\mu'-\nu,\nu) = \{v\in V_{\lambda}(\mu):e_{\alpha_i}^{c_{i}+1}v = 0\text{ for all }i\in [r]\} = V_{\lambda}(\mu).\]
We are then done by \eqref{eqn:Oct13} as
\[K_{\lambda,\mu} = \dim(V_{\lambda}(\mu)) = \dim(V_{\lambda}(\mu'-\nu,\nu)) = c^{\mu'}_{\lambda,\nu}.\qedhere\]
\end{proof}

\subsection{BZ-patterns}
\begin{defin}\label{def:BZ-pattern}
A real \emph{BZ-pattern} in type $B_r/C_r$ (where BZ stands for Berenstein-Zelevinsky) of highest weight $\lambda=\lambda_{11}e_1+\cdots+\lambda_{1r}e_r$ is an array of real numbers $\{\lambda_{i,j},\eta_{i,j}\}$ with $1\leq i\leq j\leq r$, satisfying the inequalities 
\begin{equation}\label{eq:BZ-ineq}
\min(\lambda_{i,j-1},\lambda_{i+1,j-1})\geq\eta_{i,j-1}\geq\max(\lambda_{i,j},\lambda_{i+1,j})
\end{equation}
whenever the indices are defined, and $\eta_{i,r}\geq0$ for $1\leq i\leq r$. Denote $|\lambda_i|:=\sum_{j}\lambda_{i,j}$ and $|\eta_i|:=\sum_{j}\eta_{i,j}$. We say that such a BZ-pattern has \emph{weight} $\mu=\mu_1e_1+\cdots+\mu_re_r$ where $\mu_i=|\lambda_i|+|\lambda_{i+1}|-2|\eta_i|$ with the convention that $|\lambda_{r+1}|=0$.

A real BZ-pattern in type $D_r$ of highest weight $\lambda=\lambda_{11}e_1+\cdots+\lambda_{1r}e_r$ is an array of real numbers $\{\lambda_{i,j}\}_{1\leq i\leq j\leq r}\cup\{\eta_{i,j}\}_{1\leq i\leq j<r}$, satisfying \eqref{eq:BZ-ineq} and 
\begin{equation}\label{eq:type-D}
\lambda_{i,r}+\lambda_{i+1,r}+\min(\lambda_{i,r-1},\lambda_{i+1,r-1})\geq\eta_{i,r-1}
\end{equation}
for all $i$ with the convention that $\lambda_{i+1,i}=+\infty$. The weight of a real type $D_r$ BZ-pattern is defined in the same way as in type $B_r/C_r$.
\end{defin}

It is visually more convenient to view BZ-patterns in the following way:
\[\begin{pmatrix}
\lambda_{11} & & \lambda_{12} & \cdots & \cdots & \lambda_{1r} & \\
& \eta_{11} & & \eta_{12} & \cdots & \cdots & \eta_{1r} \\
& & \lambda_{22} & \cdots & \cdots & \lambda_{2r} & \vdots \\
& & & \ddots & & \vdots & \vdots \\
& & & & & \lambda_{rr} & \vdots \\
& & & & & & \eta_{rr}
\end{pmatrix}\]
where the inequalities in \eqref{eq:BZ-ineq} dictate that adjacent entries (in the $45^{\circ}$ manner) increase from left to right.
\begin{defin}\label{def:BZ-pattern-polytope}
For a dominant weight $\lambda=\lambda_{11}e_1+\cdots+\lambda_{1r}e_r$, define $\BZ_{\lambda}$, the \emph{BZ-polytope} of highest weight $\lambda$, to be the polytope cut off by the inequalities specified in Definition~\ref{def:BZ-pattern}. For another dominant weight $\mu$, let the BZ-polytope of highest weight $\lambda$ and weight $\mu$, $\BZ_{\lambda,\mu}$, be a section of $\BZ_{\lambda}$ by imposing the equality $\mu_i=|\lambda_i|+|\lambda_{i+1}|-2|\eta_i|$ for $i=1,\ldots,r$, as in Definition~\ref{def:BZ-pattern}.
\end{defin}

The notion of $\BZ$-patterns allows us to work with classical types almost uniformly. Note that a $\BZ$-pattern of type $D$ has different coordinates and additional inequalities comparing to type $B$ and $C$.
\begin{defin}\label{def:BZ-pattern-integral}
A real BZ-pattern $\{\lambda_{i,j},\eta_{i,j}\}$ is \emph{integral} of 
\begin{itemize}
\item type $C_r$, if $\lambda_{i,j},\eta_{i,j}\in\mathbb{Z}$;
\item type $B_r$, if all of $\lambda_{i,j}$, $1\leq i\leq j\leq r$ and $\eta_{i,j}$, $1\leq i\leq j<r$ lie simultaneously in either $\mathbb{Z}$ or $\frac{1}{2}+\mathbb{Z}$, and $\eta_{i,r}\in\frac{1}{2}\mathbb{Z}$;
\item type $D_r$, if all of $\lambda_{i,j}$ and $\eta_{i,j}$ lie simultaneously in either $\mathbb{Z}$ or $\frac{1}{2}+\mathbb{Z}$.
\end{itemize}
\end{defin}

\begin{theorem}[\cite{BZ88}]\label{thm:BZ-pattern-integral}
In types $B_r,C_r,D_r$, $K_{\lambda,\mu}$ equals the number of integral BZ-patterns of highest weight $\lambda$ and weight $\mu$. 
\end{theorem}

\section{Proof of the main theorem, Theorem~\ref{thm:main}}\label{sec:proof}
We first provide the general framework to out proof, before going into each subsection for the corresponding step:
\begin{enumerate}
\item Reduce to the case where the pair $(\lambda,\mu)$ is primitive.
\item Compute the dimension of the polytope $\BZ_{\lambda}$.
\item Identify a point $P\in\BZ_{\lambda}^{\circ}\cap \BZ_{\lambda,\mu}$ to compute $\dim\BZ_{\lambda,\mu}$. 
\item Finish by establishing the equality $\deg K_{\lambda,\mu}=\dim \BZ_{\lambda,\mu}$.
\end{enumerate}
Throughout, consider a root system $\Phi$ of rank $r$ of classical type.

\subsection{Reduction to simple primitive pairs}
For any $\Pi'\subset \Pi$, let $\Phi'\subset \Phi$ be the root subsystem generated by $\Pi'$. Let $\mathfrak{g}'\subset \mathfrak{g}$ be the subalgebra defined by
\begin{equation}\label{eqn:subalg}
    \mathfrak{g}' = \mathfrak{h}'\oplus \bigoplus_{\alpha\in \Phi'} \mathfrak{g}_{\alpha},
\end{equation}
where $\mathfrak{h}'$ is spanned by $\{h_\alpha\in \mathfrak{h}:\alpha \in \Phi'\}$. Let $p:\mathfrak{h}^*\rightarrow (\mathfrak{h}')^*$ be the natural projection.
\begin{lemma}\label{lemma:weightproj}
    $p(\omega_i) = 0$ if $\alpha_i\notin \Pi'$ and $p(\omega_i)$ is a fundamental weight of $\mathfrak{g}'$ if $\alpha_i\in \Pi'$. In fact, all fundamental weights of $\mathfrak{g}'$ can be written as $p(\omega_i)$ for some $i$.
\end{lemma}
\begin{proof}
    For $\alpha_i\notin \Pi'$, $\langle \omega_i,\alpha\rangle = 0$ for all $\alpha\in \Pi'$. Since $\{\alpha\in \Pi'\}$ span $(\mathfrak{h}')^*$, $\omega_i$ is orthogonal to $(\mathfrak{h}')^*$ and $p(\omega_i) = 0$. Now for $\alpha_i\in \Pi'$, we have $\alpha_i\in (\mathfrak{h}')^*$. Therefore $\langle \omega_i,\alpha_i \rangle = \langle p(\omega_i),\alpha_i \rangle = 1$. For any $\alpha_j\in \Pi'$ such that $i\neq j$, since $\langle \omega_i,\alpha_j \rangle = 0$ and $\alpha_j\in (\mathfrak{h}')^*$, we get $\langle p(\omega_i),\alpha_j \rangle = 0$. Therefore $\langle p(\omega_i),\alpha_j \rangle = \delta_{i,j}$ and $p(\omega_i)$ is a fundamental weight of $\mathfrak{g}'$. Since the number of fundamental weights of $\mathfrak{g}'$ equals number of simple roots, $\{p(\omega_i):\alpha_i\in \Pi'\}$ are the fundamental weights of $\mathfrak{g}'$.
\end{proof}

\begin{lemma}[\cite{BZ88}, Proposition~1.3]\label{lemma:multg0}
    Let $\lambda,\nu,\mu$ be dominant integral weights of $\mathfrak{g}$ such that $\lambda+\nu-\mu\in \Phi'$. Denote $c_{p(\lambda),p(\nu)}^{p(\mu)}(\mathfrak{g}')$ the multiplicity of $V_{p(\mu)}\subset V_{p(\lambda)}\otimes V_{p(\nu)}$ as $\mathfrak{g}'$-module. Then
    \[c^{\mu}_{\lambda,\nu} = c^{p(\mu)}_{p(\lambda),p(\nu)}(\mathfrak{g}').\]
\end{lemma}

\begin{lemma}\label{lemma:Kostkag0}
    For $\lambda,\mu\in \Lambda^+$ such that $\lambda$ dominates $\mu$ and $\lambda-\mu \in \Phi'$, 
    \[K_{\lambda,\mu} = K_{p(\lambda),p(\mu)}(\mathfrak{g}'),\]
    where $K_{p(\lambda),p(\mu)}(\mathfrak{g}')$ is the weight multiplicity of $p(\mu)$ in the $\mathfrak{g}'$-module $V_{p(\lambda)}$.
\end{lemma}
\begin{proof}
    Since $\lambda$ dominates $\mu$ and $\lambda-\mu\in \Phi'$, we can write 
    \[\lambda-\mu = \sum_{\{i:\alpha_i\in \Pi'\}}c_i\alpha_i.\]
    By Lemma~\ref{lemma:Kostkatomult}, $K_{\lambda,\mu} = c^{\mu'}_{\lambda,\nu}$, where $\nu = \sum_{\{i:\alpha_i\in \Pi'\}}c_i\omega_i$ and $\mu' = \mu+\nu$. Since $\lambda+\nu-\mu' = \lambda-\mu\in \Phi'$, by Lemma~\ref{lemma:multg0}, $ c^{\mu'}_{\lambda,\nu} = c^{p(\mu')}_{p(\lambda),p(\nu)}(\mathfrak{g}')$. Since $\lambda-\mu\in (\mathfrak{h}')^*$, 
    $p(\lambda)-p(\mu) = p(\lambda-\mu) = \lambda-\mu = \sum_{\{i:\alpha_i\in \Pi'\}}c_i\alpha_i$. By Lemma~\ref{lemma:weightproj}, $p(\nu) = \sum_{\{i:\alpha_i\in \Pi'\}}c_i p(\omega_i)$ and each $p(\omega_i)$ is a fundamental weight of $\mathfrak{g}'$ that corresponds to $\alpha_i\in (\mathfrak{h}')^*$. Since $p(\mu') = p(\mu+\nu) = p(\mu)+p(\nu)$, we can apply Lemma~\ref{lemma:Kostkatomult} again and get $K_{p(\lambda),p(\mu)}(\mathfrak{g}') = c^{p(\mu')}_{p(\lambda),p(\nu)}(\mathfrak{g}')$. Therefore 
    \[K_{\lambda,\mu} =c^{\mu'}_{\lambda,\nu} = c^{p(\mu')}_{p(\lambda),p(\nu)}(\mathfrak{g}') =  K_{p(\lambda),p(\mu)}(\mathfrak{g}').\qedhere\]
\end{proof}

We now return to the context of non-primitive pairs. For a pair $(\lambda,\mu)$ that is not primitive, let $\Pi_0 = \{\alpha_i:c_i\neq 0\}$ be the subset of $\Pi$ consisting of simple roots that appears in \eqref{eqn:lambda-mu}. Let $\Phi_0\subset \Phi, W_0\subset W$ be the corresponding root subsystem and Weyl subgroup and $\mathfrak{g}_0\subset \mathfrak{g}$ the Lie subalgebra as in \eqref{eqn:subalg}. Write $r_0 = rk(\Phi_0)$ and $\{\omega_{\alpha}:\alpha\in \Pi_0\}$ the fundamental weights of $\mathfrak{g}_0$. Since $\mathfrak{g}_0$ is also semisimple, we can write $\mathfrak{g}_0 = \bigoplus_{k = 1}^m \mathfrak{g}_k$ where each $\mathfrak{g}_k$ is a simple Lie algebra. Let $\mathfrak{h}_k, \Pi_k, \Phi_k, W_k$ be the Cartan subalgebra, simple roots, root system and Weyl group of $\mathfrak{g}_k$. In particular, $\Pi_0 = \bigsqcup_{k = 1}^m \Pi_k$ and $W_0 = W_1\times \ldots \times W_m$. Let $\rho_k = \frac{1}{2}\sum_{\alpha\in \Phi^+_k}\alpha$ be the half sum of positive roots in $\Phi_k$.
Define $p_k:\mathfrak{h}^*\rightarrow (\mathfrak{h_k})^*$ to be the natural projection.
\begin{prop}\label{prop:primred}
    There exists pairs $(\lambda^{(k)},\mu^{(k)})_{k \in [m]}$ such that each $(\lambda^{(k)},\mu^{(k)})$ is a primitive pair of $\mathfrak{g}_k$ and 
    \[K_{\lambda,\mu} = \prod_{k = 1}^{m}K_{\lambda^{(k)},\mu^{(k)}}(\mathfrak{g}_k).\]
\end{prop}
\begin{proof}
    Let $\lambda^{(i)} = p_i(\lambda)$ and $\mu^{(i)} = p_i(\mu)$ for all $i\in [0,m]$. By Lemma~\ref{lemma:Kostkag0}, 
    \begin{equation}\label{eqn:Aug24bbb}
        K_{\lambda,\mu} = K_{\lambda^{(0)},\mu^{(0)}}(\mathfrak{g}_0).
    \end{equation}
    For any $k\in [m]$, set 
    \[\tilde\lambda^{(k)} = \sum_{\alpha\in \Pi_k}\langle \lambda^{(0)},\alpha \rangle \omega_{\alpha}\text{ and }\tilde\mu^{(k)} = \sum_{\alpha\in \Pi_k}\langle \mu^{(0)},\alpha \rangle \omega_{\alpha}.\] 
    Notice that $\tilde\lambda^{(k)},\tilde\mu^{(k)}\in (\mathfrak{h}_k)^*\subset \mathfrak{h}^*$ and $\lambda^{(0)} = \sum_{k=1}^{m}\tilde\lambda^{(k)}$ and $\mu^{(0)} = \sum_{k = 1}^m \tilde\mu^{(k)}$. By \eqref{thm:Kostantmult}, 
    \begin{equation*}
            K_{\lambda^{0},\mu^{(0)}}(\mathfrak{g}_0) = \sum_{w = (w_1,\ldots,w_m)\in W_0}(\prod_{k = 1}^{m}\det(w_k))P(w(\sum_{k=1}^{m}(\tilde\lambda^{(k)}+\rho_k))-\sum_{k = 1}^{m}(\tilde\mu^{(k)}-\rho_k))
    \end{equation*}
    Since $\tilde\lambda^{(k)},\tilde\mu^{(k)}$ are invariant under $W_j$ for all $j\neq k$, for $w = (w_1,\ldots,w_m)\in W$,
    \[P(w(\sum_{k=1}^{m}(\tilde\lambda^{(k)}+\rho_k))-\sum_{k = 1}^{m}(\tilde\mu^{(k)}-\rho_k)) =  P(\sum_{k = 1}^m(w_k(\tilde\lambda^{(k)}+\rho_k) - \tilde\mu^{(k)}-\rho_k)).\]
    Since $w_k(\tilde\lambda^{(k)}+\rho_k) - \tilde\mu^{(k)}-\rho_k)\in (\mathfrak{h}_k)^*$ for all $k\in [m]$, 
    \[P(\sum_{k = 1}^m(w_k(\tilde\lambda^{(k)}+\rho_k) - \tilde\mu^{(k)}-\rho_k)) = \prod_{k = 1}^m P(w_k(\tilde\lambda^{(k)}+\rho_k) - \tilde\mu^{(k)}-\rho_k).\]
    Therefore
    \begin{equation}\label{eqn:Aug24aaa}
        \begin{aligned}
            K_{\lambda^{0},\mu^{(0)}}(\mathfrak{g}_0) &= \sum_{(w_1,\ldots,w_m)\in W}\prod_{k = 1}^m (\det(w_k)P(w_k(\tilde\lambda^{(k)}+\rho_k) - \tilde\mu^{(k)}-\rho_k))\\
            &= \prod_{k = 1}^m \sum_{w_k\in W_k}(\det(w_k)P(w_k(\tilde\lambda^{(k)}+\rho_k) - \tilde\mu^{(k)}-\rho_k))\\
            &= \prod_{k = 1}^{m}K_{\tilde\lambda^{(k)},\tilde\mu^{(k)}}(\mathfrak{g}_0).
        \end{aligned}
    \end{equation}
    Define $\lambda^{(k)} = p_k(\tilde\lambda^{(k)})$ and $\mu^{(k)} = p_k(\tilde\mu^{(k)})$. Since $\lambda^{(k)}-\mu^{(k)} = p_k(\tilde\lambda^{(k)}-\tilde\mu^{(k)}) = p_k(\lambda^{(0)}-\mu^{(0)})$ and every simple root of $\mathfrak{g}_k$ appear in the expansion of $\lambda^{(0)}-\mu^{(0)}$, $(\lambda^{(k)},\mu^{(k)})$ is a primitive pair of $\mathfrak{g}_k$ for all $k\in [m]$. By Lemma~\ref{lemma:Kostkag0}, $K_{\tilde\lambda^{(k)},\tilde\mu^{(k)}}(\mathfrak{g}_0) = K_{\lambda^{(k)},\mu^{(k)}}(\mathfrak{g}_k)$. We are then done by \eqref{eqn:Aug24bbb} and \eqref{eqn:Aug24aaa}.
\end{proof}
An immediate consequence of Proposition~\ref{prop:primred} is the following: 
\begin{cor}\label{cor:reduction}
    $\deg K_{\lambda,\mu}(\mathfrak{g})(N) = \sum_{k = 1}^m \deg K_{\lambda^{(k)},\mu^{(k)}}(\mathfrak{g}_k)(N)$.
\end{cor}

By induction on $k$, it is now enough to prove Theorem~\ref{thm:main} when $(\lambda,\mu)$ is a simple primitive pair.




\subsection{Dimension of the BZ-polytope $\BZ_\lambda$}
Recall that the BZ-polytopes have coordinates indxed by 
\[\begin{cases}
 \{\lambda_{i,j}\}_{1< i\leq j\leq r} \text{ and }\{\eta_{i,j}\}_{1\leq i\leq j\leq r}\text{ for type }B_r\text{ and }C_r, \\ \{\lambda_{i,j}\}_{1< i\leq j\leq r} \text{ and }\{\eta_{i,j}\}_{1\leq i\leq j< r}\text{ for type }D_r.
\end{cases}
\]
The number of coordinates equals $|\Phi_+|$, where $\Phi$ is the root system that we are working with.
Note that since $\lambda=\sum \lambda_{1i}e_i$ is fixed, we choose not to consider $\lambda_{11},\ldots\lambda_{1r}$ as coordinates for $\BZ_{\lambda}$ and $\BZ_{\lambda,\mu}$.

\begin{lemma}\label{lem:dimBZ}
Let $\lambda=\sum_{i=1}^r d_i\omega_i$ be a dominant weight in a root system $\Phi$ where $d_i\in\mathbb{Z}_{\geq}0$ for all $i$. Let $\Phi'\subset\Phi$ be the root subsystem of $\Phi$ spanned by $\{\alpha_i\:|\: d_i=0\}$. Then we have that \[\dim\BZ_{\lambda}=|\Phi_+|-|\Phi'_+|.\]
\end{lemma}
We prove Lemma~\ref{lem:dimBZ} using two steps: upper bounding the dimension by figuring out conditions on the coordinates; and lower bounding the dimension by providing a certain neighborhood around an interior point.
\begin{proof}[Proof of Lemma~\ref{lem:dimBZ} in type $B/C$]
We adopt the conventions in Section~\ref{sec:background}. Let the simple roots of $\Phi'$ be $S'$ and let $S'=Z_1\sqcup\cdots\sqcup Z_m$ be its partition into the connected components of the Dynkin diagram of $\Phi'$ where the indices are in order. Write $z_i=|Z_i|$ for all $i$. Since $\omega_j=e_1+\cdots+e_j$ for all $j$, if $j\in S'$, $\lambda_{1,j}=\lambda_{1,j+1}$, with the convention that $\lambda_{1,r+1}=0$. 

\

\noindent\textbf{Case 1:} $r\in S'$. Here $Z_1,\ldots,Z_{m-1}$ are of type $A$ and $Z_m$ is of type $B/C$ with \[|\Phi'_+|=\sum_{i=1}^{m-1}{z_i+1\choose 2}+z_m^2.\]
For each $Z_i=\{t+1,\ldots,t+z_i\}$ where $i<m$, we have that $\lambda_{1,t+1}=\cdots=\lambda_{1,t+z_i+1}$, and thus all the ${z_i+1\choose 2}$ coordinates in the triangle below must take on this value. For $Z_m=\{r-z_m+1,\ldots,r\}\subset\Phi'$, we know that $\lambda_{1,r-z_m}=\cdots=\lambda_{1,r}=0$. Since $\eta_{i,r}\geq0$ for all $i$, we again deduce that $\lambda_{i,j}=0$, $\eta_{i,j}=0$ for $j-i\geq r-z_m-1$. There are $z_m^2$ such coordinates, which is the number of positive roots in a type $B_{z_m}$ root system. As a result, we have $|\Phi'_+|$ fixed coordinates so $\dim \BZ_{\lambda}\leq |\Phi_+|-|\Phi'_+|$.

To obtain an upper bound, we construct an interior point $T\in\BZ_{\lambda}^{\circ}$, as follows. The coordinates $\lambda_{1,1},\ldots,\lambda_{1,r}$ have been given. For all the other coordinates, take $\lambda_{i,j}=(\eta_{i-1,j-1}+\eta_{i-1,j})/2$ and $\eta_{i,j}=(\lambda_{i,j}+\lambda_{i,j+1})/2$ to be the average of the two values above it, where $\eta_{i,r}=\lambda_{i,r}/2$. See Figure~\ref{fig:typeB-interior-point} for an example.
\begin{figure}[h!]
\centering
$\begin{pmatrix}
3 && 1 && 1 && 0 && 0 & \\
& 2 && \underline{1} && 1/2 && \underline{0} && \underline{0} \\
&& 3/2 && 3/4 && 1/4 && \underline{0} & \\
&&& 9/8 && 1/2 && 1/8 && \underline{0} \\
&&&& 13/16 && 5/16 && 1/16 & \\
&&&&& \cdots &&&&
\end{pmatrix}$
\caption{An interior point in $\BZ_{\lambda}$ of type $B/C$, where $\lambda=(3,1,1,0,0)$ with $\S'=\{2,4,5\}$. The fixed values are underlined.}
\label{fig:typeB-interior-point}
\end{figure}
In the previous paragraph, we analyzed the coordinates that are fixed. On the other hand, from this construction, for each coordinate $\lambda_{i,j}$ or $\eta_{i,j}$ that is not fixed, all of the inequalities involved (see Definition~\ref{def:BZ-pattern} for the inequalities) are not tight, and we can thus find a small cube of dimension $|\Phi_+|-|\Phi'_+|$ around $T$ with side length $\epsilon>0$ that is small enough. As a result, $\dim\BZ_{\lambda}\geq |\Phi_+|-|\Phi'_+|$ as well.

\

\noindent\textbf{Case 2:} $r\notin S'$. Here, $Z_1,\ldots,Z_m$ are all of type $A$. It is a simpler case than case 1. As above, we have $\sum_{i=1}^n{z_i+1\choose 2}$ fixed coordinates coming from the triangles below each consecutive chunk of same values in the first row. We can analogously construct an interior point $T\in\BZ_{\lambda}^{\circ}$ by assigning the average value of the two values above each coordinate with the convention that $\lambda_i,r+1 = 0$ for all $i$. Both the upper and lower bounds of $\dim\BZ_{\lambda}$ are established completely analogously. So $\dim\BZ_{\lambda}= |\Phi_+|-|\Phi'_+|$ as desired.
\end{proof}

\begin{proof}[Proof of Lemma~\ref{lem:dimBZ} in type $D$]
Again, let $S'=Z_1\sqcup\cdots\sqcup Z_m$ be the partition of simple roots in $\Phi'$ into connected components. The main extra difficulty is the additional inequalities and the fact that $\lambda_{1,r}$ (and some other coordinates) can be negative since $\omega_{r-1}=\frac{1}{2}(1,1,\ldots,1,-1)$. 

\

\noindent\textbf{Case 1:} $r-1,r\notin S'$. All of the $Z_1,\ldots,Z_m$ are of type $A$. As above, we have $\sum_{i=1}^n{z_i+1\choose 2}=|\Phi'_+|$ fixed coordinates so $\dim\BZ_{\lambda}\leq|\Phi_+|-|\Phi'_+|$. 

We need some different treatment in this case to obtain an interior point $T\in\BZ_{\lambda}^{\circ}$. Notice that $\lambda_{1,r-1}=\frac{1}{2}(d_{r-1}+d_{r})$ and $\lambda_{1,r}=\frac{1}{2}(d_r-d_{r-1})$. Fix a positive number $\delta<< d_{r-1},d_r$. Set
\[\begin{cases}
\eta_{1,r-1}=\frac{1}{2}(d_{r-1}+d_r)-2\delta,\\
\lambda_{2,r-1}=\frac{1}{2}(d_{r-1}+d_r)-\delta,\\
\lambda_{2,r}=\frac{1}{2}(d_{r-1}+d_r)-3\delta,\\
\lambda_{i,j}=\frac{1}{2}(\eta_{i-1,j-1}+\eta_{i-1,j})\text{ for other }i,j\text{'s with }\lambda_{i,r}=\frac{1}{2}\eta_{i-1,r-1},\\
\eta_{i,j}=\frac{1}{2}(\lambda_{i,j}+\lambda_{i,j+1})\text{ for other }i,j\text{'s}.
\end{cases}\]
We directly see that $\eta_{1,r-2}>\lambda_{2,r-1}>\eta_{1,r-1}>\lambda_{2,r}$, where the first inequality holds as $\eta_{1,r-2}\geq\lambda_{1,r-1}=\frac{1}{2}(d_{r-1}+d_r)$, and that $\lambda_{1,r-1}>\eta_{1,r-1}>\lambda_{1,r}$. We now check the more obscure-looking inequality \eqref{eq:type-D}
\[\lambda_{i,r}+\lambda_{i+1,r}+\min(\lambda_{i,r-1},\lambda_{i+1,r-1})\geq\eta_{i,r-1}.\]
When $i=1$, we have $\min(\lambda_{1,r-1},\lambda_{2,r-1})=\lambda_{2,r-1}=\frac{1}{2}(d_{r-1}+d_r)-\delta$. The left hand side equals $\frac{1}{2}d_{r-1}+\frac{3}{2}d_r-4\delta$ while the right hand side equals $\frac{1}{2}d_{r-1}+\frac{1}{2}d_r-2\delta$, so we indeed have a strict inequality. Moreover, since $\lambda_{2,r}>0$, by the averaging construction, $\lambda_{i,j}>0$ for $i>1$ and $\eta_{i,j}>0$ for all $i,j$. The only potentially non-positive entry is $\lambda_{1,r}$. As $\min(\lambda_{i,r-1},\lambda_{i+1,r-1})\geq\eta_{i,r-1}$, this inequality \eqref{eq:type-D} is always strict when $i\geq2$ because $\lambda_{i,r},\lambda_{i+1,r}>0$. 

As a result, we now constructed an interior point $T\in\BZ_{\lambda}^{\circ}$, and all inequalities involved with non-fixed coordinates are strict. As before, there is a hypercube of dimension $|\Phi_+|-|\Phi'_+|$ around $T$ of side length $\epsilon<<\delta$. See Figure~\ref{fig:typeD-interior-point-case1} for an example of this construction.
\begin{figure}[h!]
\centering
$\begin{pmatrix}
4 && 3 && 3 && 3 && -2 \\
& 7/2 && \underline{3} && \underline{3} && 3{-}2\delta &\\
&& 13/4 && \underline{3} && 3{-}\delta && 3{-}3\delta  \\
&&& 25/8 && 3{-}\frac{1}{2}\delta && 3{-}2\delta & \\
&&&& \cdots &&&&
\end{pmatrix}$
\caption{An interior point in $\BZ_{\lambda}$ of type $D$, where $\lambda=(4,3,3,3,-2)$ with $S'=\{2,3\}$. The fixed coordinates are underlined.}
\label{fig:typeD-interior-point-case1}
\end{figure}

\

\noindent\textbf{Case 2:} $r-1\in S'$ and $r\notin S'$ Now $\lambda$ has no contribution from $\omega_{r-1}=\frac{1}{2}(1,1,\ldots,1,-1)$ and positive contribution from $\omega_{r}=\frac{1}{2}(1,1,\ldots,1)$ so all coordinates $\lambda_{1,1},\ldots,\lambda_{1,r}$ are strictly positive. Therefore, as $\min(\lambda_{i,r-1},\lambda_{i+1,r-1})\geq\eta_{i,r-1}$, inequalities \eqref{eq:type-D} always hold and are always strict, for all $i$. This is the simplest case, which follows the same argument as in case 2 of type $B/C$ (and in fact type $A$). There are $|\Phi'_+|$ fixed coordinates coming from the triangles below the same values in $\lambda_{1,1},\ldots,\lambda_{1,r}$ so $\dim\BZ_{\lambda}\leq|\Phi_+|-|\Phi'_+|$. We again use the averaging technique to create an interior point $T\in\BZ_{\lambda}^{\circ}$ and find a neighborhood of dimension $|\Phi_+|-|\Phi'_+|$ around $T$, as all inequalities involving the coordinates that are not fixed are always strict. We conclude that $\dim\BZ_{\lambda}=|\Phi_+|-|\Phi'_+|$. 

\

\noindent\textbf{Case 3:} $r-1\notin S'$ and $r\in S'$. All the $Z_i$'s are of type $A$. We have $\lambda_{1,r-1}+\lambda_{1,r}=0$ where $\lambda_{1,r-1}>0$ and $\lambda_{1,r}<0$. More specifically, $\lambda_{1,r-z_m}=\cdots=\lambda_{1,r-1}=-\lambda_{1,r}>0$. Let this value be $t=\frac{1}{2}d_{r-1}=\lambda_{1,r-1}$. The inequality \eqref{eq:type-D} says that $\lambda_{1,r-1}+\lambda_{1,r}+\lambda_{2,r}\geq\eta_{1,r-1}$, which simplifies to $t-t+\lambda_{2,r}\geq\eta_{1,r-1}$, $\lambda_{2,r}\geq\eta_{1,r-1}$. At the same time, $\lambda_{2,r}\leq\eta_{1,r-1}$ so $\lambda_{2,r}=\eta_{1,r-1}$.  If $z_m=1$, then this equality already reduces the dimension by $1$, which is the number of positive roots in a type $A_1$ root system. If $z_m>1$, we further have $\lambda_{2,r-1}+\lambda_{2,r}+\lambda_{1,r}\geq\eta_{1,r-1}$, which simplifies to $\lambda_{2,r-1}\geq t$. When $z_m>1$, $\lambda_{2,r-1}\leq\eta_{1,r-2}\leq\lambda_{1,r-2}=t$ so $\lambda_{2,r-1}=t$. As a result, in the triangle below $\lambda_{1,r-z_m},\ldots,\lambda_{1,r}$ with ${z_m+1\choose 2}$ entries, all coordinates except $\eta_{1,r-1}$ are fixed to be $t$. Since the equality $\eta_{1,r-1}=\lambda_{2,r}\in[-t,t]$ reduces the dimension by $1$, there is a decrease of ${z_m+1\choose 2}$ in the dimension. The other $Z_1,\ldots,Z_{m-1}$ are dealt with exactly as before. Therefore, $\dim\BZ_{\lambda}\leq |\Phi_+|-|\Phi_+'|$.

To obtain a point $T\in\BZ_{\lambda,\mu}^{\circ}$, we use the averaging construction $\lambda_{i,j}=\frac{1}{2}(\eta_{i-1,j-1}+\eta_{i-1,j})$ where $\lambda_{i,r}=\frac{1}{2}\eta_{i,r}$ and $\eta_{i,j}=\frac{1}{2}(\lambda_{i,j}+\lambda_{i,j+1})$ for $z_m\geq2$. In the edges case where $z_m=1$, we slightly modify $\lambda_{2,r-1}=t+\delta$ for $0<\delta<<t$, and use the averaging construction for other coordinates. Again, all inequalities involving non-fixed entries are all strict by considering $\eta_{1,r-1}=\lambda_{2,r}$ as the same variable, and we can thus find a neighborhood of dimension $|\Phi_+|-|\Phi_+'|$ around $T$ to conclude $\dim\BZ_{\lambda}\geq|\Phi_+|-|\Phi_+'|$. See Figure~\ref{fig:typeD-interior-point-case4} for this construction.
\begin{figure}[h!]
\centering
$\begin{pmatrix}
3 && 2 && 1 && 1 && 1 && -1 \\
& 5/2 && 3/2 && \underline{1} && \underline{1} && \underline{0} &\\
&& 2 && 5/4 && \underline{1} && \underline{1} && 0 \\
&&& 13/8 && 9/8 && \underline{1} && 1/2 & \\ 
&&&& 11/8 && 17/16 && 3/4 && 1/4 \\
&&&&& \cdots &&&&&&
\end{pmatrix}$
\caption{An interior point in $\BZ_{\lambda}$ of type $D$, where $\lambda=(3,2,1,1,1,-1)$ with $S'=\{3,4,6\}$. The fixed coordinates are underlined.}
\label{fig:typeD-interior-point-case3}
\end{figure}

\

\noindent\textbf{Case 4:} both $r-1$ and $r$ are in $S'$. This means $\lambda_{1,r-1}=\lambda_{1,r}=0$. If $r-2\in S'$, then $Z_m$ is of type $D_{z_m}$ where $z_m\geq3$. If $r-2\notin S'$, then $r-1$ and $r$ are in their own connected components, in which case we let $Z_m=\{r-1,r\}$ regardless. Note that a type $D_z$ root system has $z^2-z$ positive roots, and this is also true if we adopt the convention that a type $D_2$ root system is isomorphic to $A_1\oplus A_1$. In this case here, we let $r-1,r\in Z_m$, which is a root system of type $D_{z_m}$. As before, $Z_1,\ldots,Z_{m-1}$ are of type $A$ and can be dealt similarly as before. 

Now $\lambda_{1,r-z_m+1},\ldots,\lambda_{1,r}$ all equal to $0$. We then use induction on the row number $i$ to show that $\lambda_{i,j}=0$, $\eta_{i,j}=0$ for $j-i\geq r-z_m$. The base case $i=1$ is given by the assumption. If $\lambda_{i,j}=0$ for a fixed $i$ and $j\geq r-z_m+i$, then we have $\eta_{i,j}=0$ since it is between $\lambda_{i,j}$ and $\lambda_{i,j+1}$, which are both $0$. If $\eta_{i,j}=0$ for a fixed $i$ and $j\geq r-z_m+i$, then we similarly have $\eta_{i,j}\leq \lambda_{i+1,j+1}\leq \eta_{i,j+1}$ so $\lambda_{i+1,j+1}=0$ for all $j+1<r$. It suffices to argue that $\lambda_{i+1,r}=0$. By the inequality \eqref{eq:type-D}, $\lambda_{i+1,r}+\lambda_{i+1,r-1}+\lambda_{i,r}\geq\eta_{i,r-1}$, which simplifies to $\lambda_{i+1,r}\geq0$. At the same time, $\lambda_{i+1,r}\leq\eta_{i,r-1}=0$. Therefore, $\lambda_{i+1,r}=0$ as desired. The induction step goes through. The number of resulting fixed coordinates is $z_m+z_m+(z_m-1)+(z_m-1)+\cdots+1+1=z_m^2-z_m$, which is the number of positive roots in a type $D_{z_m}$ root system.

We have now $|\Phi'_+|$ fixed coordinates. To obtain $T\in\BZ_{\lambda}^{\circ}$, we apply the averaging construction $\lambda_{i,j}=\frac{1}{2}(\eta_{i-1,j-1}+\eta_{i-1,j})$, $\eta_{i,j}=\frac{1}{2}(\lambda_{i,j}+\lambda_{i,j+1})$ and notice that all inequalities involving unfixed coordinates are strict to conclude $\dim\BZ_{\lambda}=|\Phi_+|-|\Phi'_+|$. See Figure~\ref{fig:typeD-interior-point-case4}.
\begin{figure}[h!]
\centering
$\begin{pmatrix}
2 && 1 && 0 && 0 && 0 \\
& 3/2 && 1/2 && \underline{0} && \underline{0} &\\
&& 1 && 1/4 && \underline{0} && \underline{0}  \\
&&& 5/8 && 1/8 && \underline{0} & \\ 
&&&& 3/8 && 1/16 && \underline{0} \\
&&&&& \cdots &&&&
\end{pmatrix}$
\caption{An interior point in $\BZ_{\lambda}$ of type $D$, where $\lambda=(2,1,0,0,0)$ with $S'=\{3,4,5\}$. The fixed coordinates are underlined.}
\label{fig:typeD-interior-point-case4}
\end{figure}
\end{proof}

\subsection{Dimension of the BZ-polytope $\BZ_{\lambda,\mu}$}
Let $\lambda$ be a dominant weight in a root system $\Phi$ of rank $r$ as above. Consider the linear projection $\wt:\BZ_{\lambda}\rightarrow \R^r$ which sends a BZ-pattern to its weight (See Definition~\ref{def:BZ-pattern}). Denote the image of $\wt$ by $P_{\lambda}$, called the \emph{weight polytope} in the literature. Recall that we have defined the $\BZ$-polytope $\BZ_{\lambda,\mu}$ to be $\wt^{-1}(\mu)$. The following theorem is well-known, see for example \cite{besson2022weight}.
\begin{theorem}\label{thm:weight-polytope}
The image of $P_{\lambda}=\wt(\BZ_{\lambda})$, is a polytope defined by \[P_{\lambda}=\{\mu\in\R^r\:|\: \langle \lambda,\omega_i\rangle\geq\langle\mu,u\omega_i\rangle,\text{ for all }i=1,\ldots,r\text{ and }u\in W\},\] or equivalently, by \[P_{\lambda}=\mathrm{Conv}(\{u\lambda\:|\: u\in W\}),\]where $W$ is the Weyl group. Moreover, $V_{\lambda}(\mu)\neq0$ if and only if $\mu\in P_{\lambda}$ and that $\lambda-\mu$ is a nonnegative integral linear combination of the simple roots.
\end{theorem}

We need another lemma that is purely polytopal. 
\begin{lemma}[\cite{M08}, Lemma~3.1]\label{lem:projection-interior}
Given a polytope $P\in\subset\R^m$ and a linear map $\pi:\R^m\rightarrow\R^n$, we have $(\pi(P))^{\circ}\subset\pi(P^{\circ})$, where $P^{\circ}$ is the interior of $P$.
\end{lemma}

We are now ready to deduce the dimension of $\BZ_{\lambda,\mu}$.
\begin{lemma}\label{lem:dimBZlammu}
Let $(\lambda,\mu)$ be a primitive pair of dominant weights in a root system $\Phi$ of rank $r$. Then $\dim \BZ_{\lambda,\mu}=\dim\BZ_{\lambda}-r$.
\end{lemma}
\begin{proof}
Since $\mu$ is dominant, $\langle\mu,u\omega_i\rangle\leq\langle\mu,\omega_i\rangle$ for all $u\in W$ and $i=1,\ldots,r$. Since $(\lambda,\mu)$ is a primitive pair, $\lambda-\mu$ is a positive linear combination of simple roots, meaning that $\langle\lambda-\mu,\omega_i\rangle>0$ for all $i$. As a result, $\langle\lambda,\omega_i\rangle>\langle\mu,u\omega_i\rangle$ for all $u\in W$ and $i=1,\ldots,r$. All the defining inequalities of $P_{\lambda}$ in Theorem~\ref{thm:weight-polytope} are strictly for $\mu$, so $\mu\in P_{\lambda}^{\circ}$. Recall $P_{\lambda}=\wt(\BZ_{\lambda})$. By Lemma~\ref{lem:projection-interior}, $\mu\in(\wt(\BZ_{\lambda}))^{\circ}\subset\wt(\BZ_{\lambda}^{\circ})$. Thus, there exists a BZ-pattern $T\in\BZ_{\lambda}^{\circ}$ such that $\wt(T)=\mu$. Let the coordinates of $T$ be $\{\lambda_{i,j}\}\cup\{\eta_{i,j}\}$ as before.

As a quick summary, we have a point $T$ in the interior of $\BZ_{\lambda}$ with weight $\mu$. Now the polytope $\BZ_{\lambda,\mu}$ can be viewed by intersecting $\BZ_{\lambda}$ with $r$ hyperplanes $H_1,\ldots,H_r$ which are the $r$ defining equations (see Definition~\ref{def:BZ-pattern}) of the weight $\mu$ that all pass through $T$. One directly sees that these defining equations are upper triangular with respect to the coordinates $\{\lambda_{i,j}\}\cup\{\eta_{i,j}\}$ so the normals of $H_1,\ldots,H_r$ are linearly independent and $K:=H_1\cap\cdots\cap H_r$ has codimension $r$. As the linear subspace $K$ passes through an interior point of $\BZ_{\lambda}$, $\BZ_{\lambda,\mu}=K\cap\BZ_{\lambda}$ must have codimension $r$ in $\BZ_{\lambda}$. Therefore, $\dim\BZ_{\lambda,\mu}=\dim\BZ_{\lambda}-r$ as desired.
\end{proof}

\subsection{Finishing the proof}
We use the following classical result in rational Ehrhart theory.
\begin{theorem}[\cite{ehrhart,mcmullen}]\label{thm:Ehrhart}
Let $Q$ be a rational polytope, i.e. the convex hull of finitely many points in $\mathbb{Q}^n$. Then $|NQ\cap\mathbb{Z}^n|$ is a quasi-polynomial in $N$ of degree $\dim Q$, for a positive integer $N$.
\end{theorem}

We are now ready for the final step of our main proof.
\begin{lemma}\label{lem:deg=dim}
Let $(\lambda,\mu)$ be a primitive pair of dominant weight in a root system of classical type. Then $\deg K_{\lambda,\mu}(N)=\dim\BZ_{\lambda,\mu}$.
\end{lemma}
\begin{proof}
From Definition~\ref{def:BZ-pattern-polytope}, we know that $\BZ_{\lambda,\mu}$ is bounded by equalities and inequalities with integral coefficients, and is thus a rational polytope. Since all the equalities and inequalities involved do not have constant terms, $\BZ_{N\lambda,N\mu}=N\BZ_{\lambda,\mu}$ for a positive integer $N$. 

By Definition~\ref{def:BZ-pattern-integral} and Theorem~\ref{thm:BZ-pattern-integral}, in type $C_r$, $K_{\lambda,\mu}(N)=K_{N\lambda,N\mu}=|\mathbb{Z}^{|\Phi_+|}\cap N\BZ_{\lambda,\mu}|$ which is a quasi-polynomial in $N$ of degree $\dim\BZ_{\lambda,\mu}$. In type $D_r$, \[K_{\lambda,\mu}(N)=K_{N\lambda,N\mu}= \left|\mathbb{Z}^{|\Phi_+|}\cap N\BZ_{\lambda,\mu}\right|+\left|(\mathbb{Z}+\frac{1}{2})^{|\Phi_+|}\cap N\BZ_{\lambda,\mu}\right|\]
which is the sum of two quasi-polynomials in $N$ of degree $\dim\BZ_{\lambda,\mu}$ (by properly shifting the coordinate axis). As for type $B_r$, we use the same argument in type $D_r$ after scaling the coordinates $\eta_{i,r}$ by a factor of $2$, for $i=1,\ldots,r$. 
\end{proof}

Finally, the proof of Theorem~\ref{thm:main} follows from Corollary~\ref{cor:reduction}, Lemma~\ref{lem:dimBZ}, Lemma~\ref{lem:dimBZlammu} and Lemma~\ref{lem:deg=dim}.

\section*{Acknowledgements}
This project is initiated during the 2022 Graduate Research Workshop in Combinatorics (GRWC) hosted by University of Denver and CU Denver. We are really grateful for the opportunities created by this workshop. We would like to thank Tyrrell McAllister and Igor Pak for helpful conversations. SG was partially supported by NSF RTG grant DMS 1937241 and NSF Graduate Research Fellowship under grant No. DGE-1746047.


\begin{thebibliography}{10}

\bibitem{an2018fvectors}
Byung~Hee An, Yunhyung Cho, and Jang~Soo Kim.
\newblock On the {$f$}-vectors of {G}elfand-{C}etlin polytopes.
\newblock {\em European J. Combin.}, 67:61--77, 2018.

\bibitem{ardila2011Gelfand}
Federico Ardila, Thomas Bliem, and Dido Salazar.
\newblock Gelfand-{T}setlin polytopes and
  {F}eigin-{F}ourier-{L}ittelmann-{V}inberg polytopes as marked poset
  polytopes.
\newblock {\em J. Combin. Theory Ser. A}, 118(8):2454--2462, 2011.

\bibitem{BZ88}
A.~D. Berenstein and A.~V. Zelevinsky.
\newblock Tensor product multiplicities and convex polytopes in partition
  space.
\newblock {\em J. Geom. Phys.}, 5(3):453--472, 1988.

\bibitem{besson2022weight}
Marc Besson, Sam Jeralds, and Joshua Kiers.
\newblock Weight polytopes and saturation of demazure characters.
\newblock {\em arXiv preprint arXiv:2202.05405}, 2022.

\bibitem{deloera2004vertices}
Jes\'{u}s~A. De~Loera and Tyrrell~B. McAllister.
\newblock Vertices of {G}elfand-{T}setlin polytopes.
\newblock {\em Discrete Comput. Geom.}, 32(4):459--470, 2004.

\bibitem{ehrhart}
Eug\`ene Ehrhart.
\newblock Sur les poly\`edres rationnels homoth\'{e}tiques \`a {$n$}
  dimensions.
\newblock {\em C. R. Acad. Sci. Paris}, 254:616--618, 1962.

\bibitem{YoungTableaux}
William Fulton.
\newblock {\em Young tableaux}, volume~35 of {\em London Mathematical Society
  Student Texts}.
\newblock Cambridge University Press, Cambridge, 1997.
\newblock With applications to representation theory and geometry.

\bibitem{gao2019diameter}
Yibo Gao, Benjamin Krakoff, and Lisa Yang.
\newblock The diameter and automorphism group of {G}elfand-{T}setlin polytopes.
\newblock {\em Discrete Comput. Geom.}, 62(1):209--238, 2019.

\bibitem{KTT04}
R.~C. King, C.~Tollu, and F.~Toumazet.
\newblock Stretched {L}ittlewood-{R}ichardson and {K}ostka coefficients.
\newblock In {\em Symmetry in physics}, volume~34 of {\em CRM Proc. Lecture
  Notes}, pages 99--112. Amer. Math. Soc., Providence, RI, 2004.

\bibitem{mcallister2006thesis}
Tyrrell~B. McAllister.
\newblock {\em Applications of polyhedral geometry to computational
  representation theory}.
\newblock ProQuest LLC, Ann Arbor, MI, 2006.
\newblock Thesis (Ph.D.)--University of California, Davis.

\bibitem{M08}
Tyrrell~B. McAllister.
\newblock Degrees of stretched {K}ostka coefficients.
\newblock {\em J. Algebraic Combin.}, 27(3):263--273, 2008.

\bibitem{mcmullen}
P.~McMullen.
\newblock Lattice invariant valuations on rational polytopes.
\newblock {\em Arch. Math. (Basel)}, 31(5):509--516, 1978/79.

\bibitem{stanley1986two}
Richard~P. Stanley.
\newblock Two poset polytopes.
\newblock {\em Discrete Comput. Geom.}, 1(1):9--23, 1986.

\bibitem{EC2}
Richard~P. Stanley.
\newblock {\em Enumerative combinatorics. {V}ol. 2}, volume~62 of {\em
  Cambridge Studies in Advanced Mathematics}.
\newblock Cambridge University Press, Cambridge, 1999.
\newblock With a foreword by Gian-Carlo Rota and appendix 1 by Sergey Fomin.

\bibitem{Z73}
D.~P. \v{Z}elobenko.
\newblock {\em Compact {L}ie groups and their representations}.
\newblock Translations of Mathematical Monographs, Vol. 40. American
  Mathematical Society, Providence, R.I., 1973.
\newblock Translated from the Russian by Israel Program for Scientific
  Translations.

\end{thebibliography}
\end{document}